\documentclass{paper}
\usepackage[tbtags]{amsmath}
\usepackage{amsfonts,amssymb}
\usepackage[usenames]{color}

\usepackage{graphicx}
\usepackage[matrix,arrow,curve]{xy}

\newtheorem{theorem}{Theorem}
\newtheorem{definition}{Definition}

\newtheorem{lemma}{Lemma}

\newtheorem{corollary}{Corollary}
\newtheorem{proposition}{Proposition}
\newtheorem{proof}{Proof}

\begin{document}

\markboth{A.\,M.~Mikhovich}
{Quasirationality and prounipotent crossed modules}


\title{Quasirationality and prounipotent crossed modules}

\author{A.\,M.~Mikhovich}


\maketitle

\begin{abstract}
 We study quasirational presentations ($QR$-presentations) of (pro-$p$)groups,
   which contain aspherical presentations and their subpresentations, and also still
   mysterious pro-$p$-groups with a single defining relation.  Using schematization of $QR$-presentations and
   embedding of
   the rationalized module of relations into a diagram related to certain
   prounipotent crossed module, we study cohomological properties of pro-$p$-groups with a single defining relation
   (a question of J.-P. Serre).
\end{abstract}

\keywords{quasirationality, prounipotent crossed modules.}


\section{Introduction}	
In the paper \cite{Mikh2014} we introduced the notion of
quasirational (pro-$p$) presentation, which in the discrete case
contains aspherical presentations and their subpresentations,
and also pro-$p$-groups with a single defining relation in the pro-$p$-case.
Quasirational presentations naturally generalize discrete combinatorially
aspherical presentations introduced by J.~Huebschmann, and aspherical
(pro-$p$)presentations studied by O.~V.~Melnikov in \cite{Mel1}. In the paper \cite{Mikh2017}
it has been shown that the notion of existence of quasirational presentation of (pro-$p$)group $G$
is equivalent to the statement that the homology groups $H_2(G,\mathbb{Z})$ (respectively
$H_2(G,\mathbb{Z}_p)$ in the pro-$p$-case) have no torsion, and
therefore quasirationality is independent on the choice of presentation, and is the
property of the (pro-$p$)group itself. Earlier in \cite{Mikh2014}
we have proved a simple criterion of quasirationality stating that
quasirationality is equivalent to absence of torsion in the (trivial)
coinvariant module of the relations module by the action of the group.
In the discrete case, as shown in \cite{Mikh2017},
quasirationality clarifies the difference of properties of aspherical
presentation and its subpresentation, which is curious from the viewpoint of
Whitehead's celebrated ``asphericity'' conjecture. Regarding quasirational
pro-$p$-presentations, in \cite{Mikh2017} we
proved O.~V.~Melnikov's Conjecture ``on existence of envelope''
of the class of aspherical pro-$p$-presentations, pointing out that quasirational
pro-$p$- presentations correspond to all requirements of the Conjecture stated in 1997.

In this paper we shall give the description, announced in \cite{Mikh2015}, of modules of relations of
quasirational pro-$p$-presentations by means of affine group schemes technique.
For these purposes, after recalling necessary constructions in Section~2, in Section~3 we construct
a prounipotent presentation \eqref{5} from a finite presentation of a pro-$p$-group \eqref{2}
by means of $\mathbb{Q}_p$-prounipotent completion (Definition \ref{d4}) of
finitely generated free pro-$p$-groups (``schematization''). Using the analogy with the discrete and
profinite cases, we study the prounipotent analogs of 2-reduced free simplicial groups,
crossed and pre-crossed modules. Proposition \ref{p1} shows that the prounipotent crossed module
constructed from a prounipotent presentation is free (Definition \ref{d2}).

In discrete and profinite algebraic homotopy theories \cite{BH, Por} the main
benefits from systematical development of the theory of crossed modules are obtained after their
abelianizations. Such approach is effective also for prounipotent crossed modules.
Proposition \ref{p2}, Lemma \ref{l4}, Lemma \ref{l5} show that the introduced
abelianizations have structures of topological modules (Definition \ref{d3}), and in Theorem \ref{t01}
and Corollary \ref{s0} we include such objects into a commutative diagram, which seems an analog
of the Gaschutz theory \cite{Gru} for quasirational presentations of pro-$p$-groups.
Since in Lemma \ref{l4} the rationalized relation module $\overline{R} \widehat{\otimes}\mathbb{Q}_p=\varprojlim_n
R/[R,R \mathcal{M}_n]\otimes \mathbb{Q}_p$ of a presentation \eqref{eq1}
is identified with Abelianization of a continuous prounipotent completion of $R$,
then $\overline{R} \widehat{\otimes}\mathbb{Q}_p$ is included into a commutative diagram \eqref{eq2}.
Theorem \ref{t01} and Corollary \ref{s0} should be considered also as variations of ideas of
comparison of homotopy types \cite{Pri2012} in dimension 2.
We apply the obtained results to the following question of J.-P. Serre from a remark to the \cite{Se}.

Let $G_r=F/(r)_F,$ where $(r)_F$ is the normal closure of $r\in F^p[F,F]$ in a free pro-$p$-group of finite rank $F$,
then J.-P. Serre asks the following:
``Can it be true that $cd(G_r)=2,$ only if $G_r$ is torsion free (and $r\neq1$)?''

Let us first note that pro-$p$-groups are $\mathbb{F}_p$-points of prounipotent affine group schemes
defined over $\mathbb{F}_p$ - the prime field of characteristics $p\geq 2$. Actually, consider the complete group algebra
$\mathbb{F}_pG=\varprojlim \mathbb{F}_p[G_{\alpha}]$ of a pro-$p$-group
$G=\varprojlim G_{\alpha},$ where $G_{\alpha}$ are finite $p$-groups. Each group algebra
$\mathbb{F}_p[G_{\alpha}]$ is obviously a cocommutative Hopf algebra over the field $\mathbb{F}_p$.
Then the dual Hopf algebra $\mathbb{F}_p[G_{\alpha}]^*$ (for details see the beginning of the third
part of the paper) is a finitely generated commutative Hopf algebra, and hence it
determines certain affine algebraic group scheme. Let $\mathcal{G}$ be the functor of group like
elements of a Hopf algebra. Note that
$G_{\alpha}=\mathcal{G}\mathbb{F}_p[G_{\alpha}]\cong Hom_{Alg_{\mathbb{F}_p}}(\mathbb{F}_p[G_{\alpha}]^*,\mathbb{F}_p)$
\cite[Proposition 18]{Vez}, where $Hom_{Alg_{\mathbb{F}_p}}(\mathbb{F}_p[G_{\alpha}]^*,\_)$
is the functor from the category of commutative $\mathbb{F}_p$-algebras with unit into the category of sets
which assigns to each commutative $\mathbb{F}_p$- algebra $A$ with unit the set
$Hom_{Alg_{\mathbb{F}_p}}(\mathbb{F}_p[G_{\alpha}]^*,A)$ of homomorphisms
$\phi:\mathbb{F}_p[G_{\alpha}]^*\rightarrow A$ of commutative $\mathbb{F}_p$-algebras with unit.
But $$G\cong\mathcal{G}\mathbb{F}_pG\cong \varprojlim\mathcal{G}\mathbb{F}_p[G_{\alpha}]\cong \varprojlim
Hom_{Alg_{\mathbb{F}_p}}(\mathbb{F}_p[G_{\alpha}]^*,\mathbb{F}_p)\cong Hom_{Alg_{\mathbb{F}_p}}(\mathbb{F}_p[G]^*,
\mathbb{F}_p).$$

Dimension shifting enables one to compute cohomology of a (pro-$p$)group $H$ as invariants of certain modules.
From this viewpoint, in order to look like (cohomologically) a group with elements of finite order,
it suffices for the elements of a group $H$ to act as if they were of finite order,
although they can actually be not so. In \cite{Mikh2017}, using multiplication of defining relations
$r=y^p$ in a free pro-$p$-group by elements $\zeta^p$ of special kind,
we managed to obtain elements of a free pro-$p$-group $F$ of the form $y^p\cdot \zeta^p$ which
act on finite dimensional modules of arbitrarily high given dimension exactly as the initial
relation, but are not $p$-th powers themselves.
A more detailed description of cohomology of $\mathbb{F}_p$-affine group schemes
(based not only on presence or absence of elements of finite order at $\mathbb{F}_p$-points)
should be expected from the study of the Frobenius homomorphism on the algebra of regular functions.

If a finitely generated discrete nilpotent group has no torsion, then by the known Malcev theorem
it is embedded \cite[4]{Vez} into its own rational prounipotent completion.
For $\mathbb{F}_p$-prounipotent groups, for which pro-$p$-groups are their $\mathbb{F}_p$-points,
it would be too optimistic to hope for a similar statement, but we shall show that existence of an embedding into a similar type of completion implies that cohomological dimension of such a group is less or equal to 2.
Using a description of the relations module of a prounipotent group with one defining relation \cite[Corollary 12]{Mikh2015}
we point out (Proposition \ref{p01} and Corollary \ref{s4}),
in terms of continuous prounipotent completion, the condition under which a finitely
generated pro-$p$-group with one defining relation has cohomological dimension 2.
The proof is based on Theorem \ref{t01} and Corollary \ref{s0}.
From our viewpoint this condition contains a variation of J.-P. Serre's question with a positive answer.

\section{Schematization and quasirationality}
\label{subsec1}
By definition, a finite type presentation of a discrete group $G$ is an exact sequence

\begin{equation}
1 \rightarrow R \rightarrow F \xrightarrow{\pi} G \rightarrow 1 \label{eq1}
\end{equation}
in which $F=F(X)$ is the free group with a finite set $X$ of generators, and $R$
is a normal subgroup in $F$ generated by a finite number of defining relations $r\in R$.
By a pro-$p$-group one calls a group isomorphic to projective limit of finite $p$-groups.
This is a topological group (with the topology of direct product) which is a compact totally disconnected group. For such groups one has a presentation theory
which is in many aspects similar to the combinatorial theory of discrete groups \cite{Koch}, \cite{ZR}.
By analogy with finite type presentation of a discrete group, we shall say that a pro-$p$-group $G$
is given by a finite type pro-$p$- presentation if $G$ is included into an exact sequence \eqref{eq1}
in which $F$ is a free pro-$p$-group with finite number of generators, and $R$
is a closed normal subgroup topologically generated by a finite number of elements in $F$,
contained in the Frattini subgroup of the group $F$ \cite{Koch}, \cite{ZR}.

Let $R=\varprojlim R_{\alpha}$ be a profinite ring ($R_{\alpha}$ are finite rings),
then denote by $RG$ the completed group algebra of a pro-$p$-group $G$.
By a completed group algebra we understand the topological algebra $RG=\varprojlim RG_{\mu}$ \cite[5.3]{ZR},
where $G=\varprojlim G_{\mu}$ is a decomposition of the pro-$p$-group $G$ into the projective limit of
finite $p$-groups $G_{\mu}$.

For discrete groups, $p$ will run over all primes, while for pro-$p$- groups $p$ is fixed. Let $G$
be a (pro-$p$)group with a finite type (pro-$p$)presentation \eqref{eq1}, and
$\overline{R}=R/[R,R]$ be the corresponding \emph{$G$-module of relations}, where $[R,R]$ is the commutant,
and the action of $G$ is induced by conjugation by $F$ on $R$. For each prime number $p\geq2$
denote by ${\Delta}_p$ the augmentation ideal of the ring $\mathbb{F}_pG.$ In the pro-$p$-case
by ${\Delta}^n$ we understand the closure of the module generated by the $n$-th powers of elements from
$\Delta={\Delta}_p$, and in the discrete case this is the $n$-th power of the ideal ${\Delta}_p$ \cite{Pas}.
The properties of this filtration in the pro-$p$-case are exposed in \cite[7.4]{Koch},
and in the discrete case the properties of the Zassenhaus filtration are similar \cite[Chap.11]{Pas},
the difference is in the use of the usual group ring instead of the completed one.

Denote by $\mathcal{M}_n, n\in \mathbb{N}$ its Zassenhaus filtration in $F$
with coefficients in the field $\mathbb{F}_p$, which is defined by the rule
$\mathcal{M}_{n,p}=\{f \in F\mid f-1 \in {\Delta}^n_p\}.$
We shall denote these filtrations simply by $\mathcal{M}_n$, omitting the prime $p$,
since its choice will be always clear from the context.
Introduce the notation $\mathbb{Z}_{(p)}$ for $\mathbb{Z}$ in the case of discrete groups and
$\mathbb{Z}_{p}$ in the case of pro-$p$-groups.

\begin{definition} We shall call the (pro-$p$)presentation \eqref{eq1} quasirational ($QR$-(pro-$p$) presentation)
if one of the following three equivalent conditions is satisfied:

(i) for each $n>0$ and for each prime $p\geq2$, the $F/R\mathcal{M}_n$-module $R/[R,R\mathcal{M}_n]$
has no $p$-torsion ($p$ is fixed for pro-$p$-groups and runs over all primes $p\geq2$
and the corresponding Zassenhaus $p$-filtrations in the discrete case);

(ii) the quotient module of coinvariants $\overline{R}_G=\overline{R}_F=R/[R,F]$ has no torsion;

(iii) $H_2(G,\mathbb{Z}_{(p)})$ has no torsion.
\end{definition}
The proof of equivalence of conditions (i) -- (iii) is contained in \cite[Proposition ~4]{Mikh2014} and
\cite[Proposition 1]{Mikh2017a}. $QR$-presentations are curious in particular due to the fact that
they contain aspherical presentations of discrete groups and their subpresentations,
and also pro-$p$-presentations of pro-$p$-groups with one defining relation \cite{Mikh2014}.

By an affine group scheme over a field $k$ one calls a representable $k$-functor $G$ from the
category $Alg_k$ of commutative $k$-algebras with unit to the category of groups.
If $G$ is representable by an algebra $\mathcal{O}(G)$, then for any commutative $k$-algebra $A$ the functor $G$ is given
by the formula
$$G(A) = Hom_{Alg_k}(\mathcal{O}(G),A).$$
Of course, we assume that the considered homomorphisms $Hom_{Alg_k}$ map the unit of the algebra $\mathcal{O}(G)$
to the unit of $A$. The algebra $\mathcal{O}(G)$ representing the functor $G$ is usually called the
\emph{algebra of regular functions} of $G$. The Yoneda lemma implies anti-equivalence
of the categories of affine group schemes and commutative Hopf algebras (with unit) \cite[1.3]{Wat}.
Let us say that an affine group scheme $G$ is \emph{algebraic} if its coordinate Hopf algebra $\mathcal{O}(G)$
is a finitely generated algebra.


Let $A$ be a coalgebra over a field $k$ with the coproduct $\Delta:A\rightarrow A\otimes A$ and the counit
$\varepsilon: A\rightarrow k$. Let us say that $A$ is conilpotent (or ``connected'' in the terminology of
\cite[B.3]{Qui5}) if there exists an element $1_A\in A$ such that $\varepsilon(1_A)=1_k,$ where $1_k$
is the unit of the field $k$, $\Delta(1_A)=1_A\otimes 1_A$, and there exists $n\in \mathbb{N}$
such that $A=F_nA,$ where $F_rA$ is the filtration in $A$ defined recursively by the following formulas:
$F_0A=k\cdot 1_A, F_rA=\{x\in A\mid \Delta x - x\otimes 1_A - 1_A\otimes x \in F_{r-1}\otimes F_{r-1}\}.$

\begin{definition} \label{d5} By a unipotent group one calls an affine algebraic group scheme
$G$ whose Hopf algebra of regular functions $\mathcal{O}(G)$ is conilpotent (see \cite[8]{Wat} and
\cite[Proposition 16]{Vez}, where equivalent definitions are given). An affine group scheme
$G=\varprojlim G_{\alpha},$ where $G_{\alpha}$ are affine algebraic group schemes
over a field $k$, is called a prounipotent group if each $G_{\alpha}$ is a unipotent group.
\end{definition}

Let $G$ be a prounipotent group with the algebra of regular functions $\mathcal{O}(G)$, then
the unipotency condition  $G_{\alpha}$ in Definition \ref{d5} is equivalent to the condition that
the so called conilpotent filtration
$0\subset C_0=I^{\perp}\subset C_1=(I^2)^{\perp}\subset...\subset C_k=(I^{k+1})^{\perp}\subset...$
is exhausting (such coalgebras are usually called locally conilpotent), i.~e. $\mathcal{O}(G)=\varinjlim C_i,$
where, as usual, $I$ is the augmentation ideal (the kernel of the counit) in the Hopf algebra
$\mathcal{O}(G)^*,$ and $(I^i)^{\perp}=\{r\in \mathcal{O}(G):r(\phi)=0,\forall\phi\in I^i\}$ (see
Section \ref{subsec2} regarding the definition of $\mathcal{O}(G)^*$ and the duality questions).

Let $A$ be a Hopf algebra over a field $k$ of characteristic 0, in which: 1) the product
is commutative; 2) the coproduct is conilpotent. Then, as an algebra, $A$
is isomorphic to a free commutative algebra \cite[Theorem 3.9.1]{Car}.
Therefore, each unipotent group $G_{\alpha}$ is isomorphic as an algebraic variety to certain
$n_{\alpha}$-dimensional affine space $\mathbb{A}^{n_{\alpha}}_k$ \cite[Theorem 4.4]{Wat}
and hence it is an affine algebraic group, and consequently \cite[Corollary 4.4]{Wat} a linear algebraic group.
Thus, we can use results from the theory of linear algebraic groups in characteristic 0.

Also one has a well known correspondence between unipotent groups over $k$
and nilpotent Lie algebras over $k$, which assigns to a unipotent group its Lie algebra.
This correspondence is easily extended to a correspondence between prounipotent groups over $k$
and pronilpotent Lie algebras over $k$ \cite[Appendix А]{Qui5}.
Functoriality of this correspondence enables one to interpret, when it is convenient, problems on prounipotent groups
in the language of Lie algebras. For example, the image of a closed subgroup
under a homomorphism of pro\-uni\-pot\-ent groups will be always a closed subgroup.

The group of $\mathbb{Q}_p$-points of any affine group scheme $G$ over $\mathbb{Q}_p$ has the $p$-adic topology.
Actually, \cite[Part 2]{DM} shows that $G$ can be represented as a projective limit
$G=\varprojlim G_{\alpha}$
of a surjective projective system of linear algebraic groups. Each $G_\alpha(\mathbb{Q}_p)$ has a canonical
$p$-adic topology induced by an embedding $G_{\alpha}\hookrightarrow GL_n$.
Let us define a topology on $G(\mathbb{Q}_p)$ as the topology of the projective limit
$G(\mathbb{Q}_p)=\varprojlim G_{\alpha}(\mathbb{Q}_p).$

\begin{definition}\label{d4}\cite[A.2.]{HM2003}, \cite{Pri2012}
Let us fix a group $G$ (with pro-$p$- topology). Define the (continuous)
prounipotent completion of $G$ as the following universal diagram, in which $\rho$ is a
(continuous) Zariski dense homomorphism from $G$ to the group of $\mathbb{Q}_p$-points
of a prounipotent affine group $G_w^{\wedge}$:
$$\xymatrix @R=0.5cm{
                &         G^{\wedge}_w(\mathbb{Q}_p)  \ar[dd]^{\tau}     \\
 G \ar[ur]^{\rho} \ar[dr]_{\chi}                 \\
                &         H(\mathbb{Q}_p)              }$$
We require that for each continuous and Zariski dense homomorphism
$\chi$ there exists a unique homomorphism $\tau$ of prounipotent groups making the diagram commutative.
\end{definition}
Such an object always exists. Indeed, let $G=\varprojlim G_{\alpha}.$
By the Zariski closure of a subgroup $K$ in $G(k)$
we understand the least proalgebraic $k$-subgroup $H\leq G$ such that $H(k)\geq K$.
It is evident that $H=\varprojlim_{\alpha}H_{\alpha},$ where $H_{\alpha}$ is the Zariski closure of
$K_{\alpha}$ in $G_{\alpha}(k).$ Consider the set of pairs $(\phi_{\alpha},U_{\alpha}),$ where
$\phi_{\alpha}: G\rightarrow U_{\alpha}(\mathbb{Q}_p)$ is a continuous homomorphism.
This is a projective system, since the pair
$$(\phi_{\alpha}\times\phi_{\beta},\overline{\phi_{\alpha}(G)\times\phi_{\beta}(G)})=(\phi_{\alpha\beta},U_{\alpha\beta})
\succeq (\phi_{\alpha},U_{\alpha}),(\phi_{\beta},U_{\beta}).$$
Here for any $\alpha,\beta$ we have denoted by
$\overline{\phi_{\alpha}(G)\times\phi_{\beta}(G)}$ the Zariski closure of the direct product
$\phi_{\alpha}(G)\times\phi_{\beta}(G)$ in $U_{\alpha}(\mathbb{Q}_p)\times U_{\beta}(\mathbb{Q}_p)$).
The partial ordering is given by the rule $(\phi_{\alpha},U_{\alpha})\succeq (\phi_{\beta},U_{\beta})$
if there exists $U_{\alpha}\rightarrow U_{\beta}$ such that
$U_{\alpha}(\mathbb{Q}_p)\twoheadrightarrow U_{\beta}(\mathbb{Q}_p)$ is an epimorphism compatible
on the images of $G$.
Now put $(\phi,G^{\wedge}_w)=\varprojlim_{\gamma} (\phi_{\gamma},U_{\gamma}).$
It is not difficult to see that the prounipotent completion of a free group $F(X)$
satisfies the universal properties inherent to a free object,
and we shall, by analogy with the discrete or pro-$p$ cases, call such a group free
and denote it by $F_u(X)$.

\section{Prounipotent crossed modules} \label{subsec2}

One has the following standard chain of equivalences:
the category of simplicial proalgebraic affine groups over a field $k$
is equivalent to the category of affine simplicial group schemes \cite[2.1]{DM},
which are dual to cosimplicial commutative Hopf algebras \cite[1.6]{Milne},
and the latter are dual to simplicial cocommutative linearly compact (pro-finite dimensional) Hopf algebras.
Let us describe the latter duality in more detail.

For an arbitrary discrete Hopf $k$-algebra $A$, the dual Hopf algebra $A^{\ast}=Hom_k(A,k)$
is a linearly compact Hopf algebra, i.~e. it is representable as the inverse limit
$A^{\ast}\cong\varprojlim A^{\ast}/U^{\perp},$ where $U$ are finite dimensional subspaces in $A$
\cite[1.5.33]{DNR}. The basis of the system of neighborhoods of zero in $A^{\ast}$ is formed by the
$k$-subspaces
$$U^{\perp}=\{\phi \in A^{\ast} \mid \phi(U)=0, \mbox{where $U$ is a finite dimensional $k$-subspace of $A$} \}.$$
For any linearly compact space $V$ it makes sense to speak about the dual space $V^\vee$
of continuous linear functions, and the evaluation map defines the Pontryagin duality \cite[1.2]{Die}
(here $V$ is a discrete or linearly compact space)
$$e: V\rightarrow V^{*\vee}, \quad v\mapsto(\phi \rightarrow \phi(v)).$$
Unifying what was said above, we obtain the following Corollary \cite[\S2,14]{Die}:

\begin{corollary}
Assigning $V\mapsto V^*$ yields a one-to-one correspondence between the structures of
(commutative) Hopf algebras on a vector $k$-space $V$ with discrete topology and the
structures of (cocommutative) linearly compact Hopf algebras on $V^*$.
\end{corollary}

From this moment and below we shall consider only either discrete commutative or
linearly compact cocommutative Hopf algebras.

First, following \cite[A.2]{Qui5}, let us state the notion of a complete augmented algebra ($CAA$)
over a field $k$. By a filtration of an algebra $A$ we shall understand a decreasing sequence of
subspaces $A=F_0A \supset F_1A \supset...$ such that $1\in F_0$ and $F_n\cdot F_m\subset F_{m+n}.$
In this case each space of the filtration is a two-sided ideal in $A$, and
$grA=\oplus_{i=0}^\infty gr_i A= \oplus_{i=0}^\infty F_iA/F_{i+1}A$ has the natural structure
of a graded algebra.

An augmented algebra $A$ with the augmentation ideal $I$ and a filtration $F_nA$
is called a complete augmented algebra if the following conditions are satisfied:

1) $\widehat{A}=\varprojlim A/I^n$ (i.~e. $A$ is complete in the $I$-adic topology);

2) the algebra $grA$ is generated by $gr_1A$;

3) $F_1A=I.$

Quillen notes that condition 2), taking into account conditions 1) and 3) from the definition of $CAA$,
is equivalent to the requirement that $F_nA$ coincides with the closure of $I^n$ in the $I$- adic topology.

Recall that the completed tensor product $E \widehat{\otimes}_k F$ of topological $k$- vector spaces $E$ and $F$
is the completion of $E \otimes_k F$ with respect to the topology (called the topology of tensor product)
given by the fundamental system of neighborhoods of 0 consisting of the sets
$V \otimes_k F + E\otimes_k W$, where $V$ (respectively $W$) is an arbitrary element of
the fundamental system of neighborhoods of 0 consisting of vector subspaces of $E$ (respectively $F$) \cite[1.2.4]{Die}.

\begin{definition} Let us say that a complete augmented linearly compact algebra $A$ is a
complete Hopf algebra (below CHA for short) if $A$ is endowed with a map
$\triangle:A \rightarrow A\widehat{\otimes} A$ of complete linearly compact algebras,
which is included into the cocommutativity and coassociativity diagrams, and has the augmentation
$A\rightarrow k$ as a counit.
 \end{definition}

Our definition is somewhat different from the definition of $CHA$ in Quillen's paper cited above,
since we additionally require linear compactness, having in mind that $CHA$ arises as a dual of Hopf
algebra.

It is convenient to control linear compactness in the definition of $CHA$, following Quillen \cite[A.3]{Qui5},
using the augmentation ideal $I$, assuming $dim_k I/I^2<\infty$.
In this situation one can construct the prounipotent completion of a finitely
generated discrete group explicitly. Let $kG$ be the group ring of a finitely generated
discrete group $G$ over certain field $k$, and let $I$ be the kernel of augmentation
(the fact that $G$ is finitely generated yields the fact that the quotients $kG/I^n$ have finite dimension).
Taking into account continuity of the coproduct $\triangle:kG\rightarrow kG\otimes kG$ in the Hopf algebra
$kG$ with respect to the $I$-adic topology (the topology on $\otimes$ is defined using the filtration
$F_n(kG\otimes kG) = \sum_{l+j=n} I^l\otimes I^j \subset kG\otimes kG$), using the $I$-adic completion
we obtain the continuous coproduct
$\widehat{\triangle}: \widehat{k}G\rightarrow \widehat{k}G\widehat{\otimes }\widehat{k}G\cong \widehat{kG\otimes kG},$
determining the structure of a complete linearly compact cocommutative Hopf algebra on
$\widehat{k}G.$ A direct check shows that the prounipotent group with the algebra of regular functions
$\widehat{k}G^\vee$ is the prounipotent completion of $G$ in the sense of Definition \ref{d4}.

In simplicial group theory, by analogy with glueing two-dimencional cells,
it is convenient to identify presentation \eqref{eq1} with the second step of construction
of free simplicial (pro-$p$)resolution of a (pro-$p$)group $G$ by the method ``pas-$\grave{a}$-pas''
going back to Andre \cite{And,Mikh2012}:

\begin{equation}
\xymatrix{
& \ar@<1ex>[r]\ar@<-1ex>[r] \ar[r] & {F(X\cup Y)} \ar@<0ex>[r]^{d_0}\ar@<-2ex>[r]^{d_1} &
F(X) \ar@<-2ex>[l]_{s_0} \ar[r] & G,} \label{2}
\end{equation}
here $d_0, d_1, s_0$ for $x \in X, y \in Y, r_y \in R$ are defined by the identities
$d_0(x)=x,  d_0(y)=1, d_1(x)=x,  d_1(y)=r_y, s_0(x)=x.$

Recall \cite{And, Mikh2012} that \eqref{2} is a free simplicial (pro-$p$) group of finite type,
degenerate in dimensions greater than two. If a pro-$p$-presentation \eqref{eq1} is minimal then
$$|Y|=dim_{\mathbb{F}_p}H^2(G,\mathbb{F}_p),|X|=dim_{\mathbb{F}_p}H^1(G,\mathbb{F}_p).$$
Let us assign to the simplicial finite type presentation \eqref{2} a presentation in the category
of complete Hopf algebras.
First, let us consider the corresponding diagram of group rings,
$\xymatrix{kF(X\cup Y)\ar[r]^{d_0} \ar@<-2ex>[r]^{d_1} & kF(X) \ar@<-2ex>[l]_{s_0}
}$.
Then we obtain from \eqref{2} using the $I$-adic completion,
taking into account finite generation of groups, the following diagram of complete
linearly compact Hopf algebras:
$$\xymatrix{\widehat{k}F(X\cup Y)\ar[r]^{d_0} \ar@<-2ex>[r]^{d_1} & \widehat{k}F(X) \ar@<-2ex>[l]_{s_0}
},$$
where $\widehat{k}F(X)= \varprojlim kF(X)/I^n,$ and $I$ is the augmentation ideal in $kF(X)$.
Applying the Pontryagin duality and the antiequivalence of the categories of commutative
Hopf algebras and affine group schemes, we obtain a diagram of free prounipotent groups
$$\xymatrix{F_u(X\cup Y)\ar[r]^{d_0} \ar@<-2ex>[r]^{d_1} & F_u(X) \ar@<-2ex>[l]_{s_0}}.$$
The fact that simplicial identities hold in \eqref{2} implies similar identities for the obtained
diagram of prounipotent groups, which is implicitly used in all the constructions of the paper.
For $k=\mathbb{F}_p$ the constructions yield pro-$p$-completions and the homotopy theory developed in
\cite{Mikh2012}, and also to the concepts of $p$-adic homotopy theory.

\begin{definition}
Let us say that we are given a fiite type presentation of a prounipotent group $G_u$ if
there exist finite sets $X$ and $Y$ such that $G_u$ is included into the following diagram
of free prounipotent groups:

\begin{equation}
\xymatrix{
{F_u(X\cup Y)} \ar@<0ex>[r]^{d_0}\ar@<-2ex>[r]^{d_1} & F_u(X) \ar@<-2ex>[l]_{s_0} \ar[r] & G_u},\label{5}
\end{equation}
in which the identities similar to \eqref{2} and $$G_u\cong F_u(X)/d_1(Kerd_0)$$ hold.
Denote $R_u=d_1(Kerd_0)$; this is a normal subgroup in $F_u(X)$,
and hence we obtain an analog of the notion of presentation \eqref{eq1} for a prounipotent group $G_u,$
to which we shall refer also as to a presentation of type \eqref{eq1}.
\end{definition}

By the schematic homotopy type of a pair
$$(X,k\mid \mbox{$X$ is a connected simplicial set},\mbox{$k$ is a field)}$$ one calls a simplicial
$k$- proalgebraic group $GX^{\wedge}_{alg_k}$ \cite[p.655]{KPT}, \cite{Pri2008}, where
$G$ is the Kan functor \cite[page 13--14]{Mikh2012}, and $GX^{\wedge}_{alg_k}$ is the
componentwise $k$- proalgebraic complement of the free simplicial group $GX$ \cite[Def. 1.6]{Pri2008}.
The constructions above are particular cases of the notion of scheme homotopy type,
expressed here in a concrete form of prounipotent completions of
finite type presentations of (pro-$p$) groups.

Below we shall work in the category of prounipotent groups over the field $k=\mathbb{Q}_p$.
By an action from the left of an affine group scheme $G_1$ on an affine group scheme $G_2$
one understands a natural transformation of functors $G_1\times G_2\rightarrow G_2$
included into the standard action diagrams \cite[6n]{Milne}.

\begin{definition} \label{d1}
By a prounipotent pre-crossed module one calls a triple $(G_2,G_1,\partial)$, where $G_1,G_2$
are prounipotent groups, $\partial: G_2 \to G_1$ is a homomorphism of prounipotent groups, $G_1$ acts on $G_2$
from the left, satisfying, for any $\mathbb{Q}_p$-algebra $A$, the identity
$$\text{CM 1) } \partial({}^{g_1}g_2) = g_1 \partial(g_2) g_1^{-1},$$
where the action is written in the form $(g_1,g_2) \to {}^{g_1}g_2$, $g_2 \in G_2(A), g_1 \in G_1(A).$
\end{definition}

\begin{definition}
A prounipotent pre-crossed module is called crossed if  for any $\mathbb{Q}_p$-algebra $A$ the following additional identity holds:
$$\text{CM 2) } ^{\partial(g_2)}g_2'= g_2 g_2' g_2^{-1}.$$
The identity CM 2) is called the Peifer identity.
Such prounipotent crossed module will be denoted by $(G_2,G_1,\partial)$.
\end{definition}

\begin{definition}
By a morphism of prounipotent crossed modules
$$(G_2,G_1, \partial) \to (G'_2,G'_1, \partial')$$ one calls a pair of homomorphisms
$\varphi: G_2 \to G'_2 \text{ and } \psi: G_1 \to G'_1$ of prounipotent groups such that
 $\varphi(^{g_1}g_2) = ^{\psi(g_1)}\varphi (g_2) \text{
and } \partial' \varphi(g_2) = \psi \partial(g_2).$
\end{definition}

For us the main role will be played by prounipotent crossed modules arising from finite
presentations of (pro-$p$)groups. Recall the known construction of crossed modules.

Assume we are given a (pro-$p$)presentation \eqref{eq1} in the simplicial form \eqref{2}.
The group $F$ acts on $R$ by conjugation $g \mapsto ^fg = f g f^{-1}, \, g \in
R, \, f \in F.$
Let us construct a pre-crossed module of this presentation. To this end,
let us first pass to the beginning of the Moore complex of the simplicial group \eqref{2}:
$$Kerd_0=NF_1\xrightarrow{d_1} NF_0,$$
where $NF_0=F(X), NF_1=Kerd_0.$
Note that both in the discrete case \cite[Proposition 1]{BH} and in the pro-$p$-case \cite[8.1.3]{ZR}
one has $NF_1=Kerd_0\cong F(Y\times F(X))$. Now we can construct the free pre-crossed module of the
(pro-$p$) presentation \eqref{2} on the set $Y$,
$$F(Y\times F(X))\xrightarrow{d_1} F(X),$$
where the action of $F(X)$ is given by ${}^fa=s_0(f)as_0(f)^{-1},$ where $f \in F(X), a \in F(Y \times F(X))$.
A detailed study of this construction in the pro-$p$ and discrete cases can be found in \cite[Proposition 12]{Mikh2012},
\cite[3]{BH}.

Below we shall define the notion of free (pre-)crossed module for prounipotent presentations of the form \eqref{5}.
\begin{definition} \label{d2}
A prounipotent (pre-)crossed module $(G_2,G_1,d)$ is called a free prounipotent pre-crossed module with the base
$Y\in G_2(\mathbb{Q}_p)$ if $Y$ generates a Zariski dense $G_1(\mathbb{Q}_p)$- subgroup in $G_2(\mathbb{Q}_p)$
with the following universal property:
for any prounipotent (pre-) crossed module $(G'_2,G'_1,d')$ and a subset $\nu(Y) \in G'_2(\mathbb{Q}_p)$,
for a function $\nu:Y \rightarrow G'_2(\mathbb{Q}_p),$ with a Zariski dense
$G'_1(\mathbb{Q}_p)$- group closure, and for any
epimorphism of prounipotent groups $f:G_1\twoheadrightarrow G'_1$ such that
$fd(\mathbb{Q}_p)(Y)=d'(\mathbb{Q}_p)\nu(Y)$, there exists a unique homomorphism
of prounipotent groups $h:G_2\rightarrow G'_2$ such that $h(\mathbb{Q}_p)(Y)=\nu(Y)$
and the pair $(h,f)$ is a homomorphism of (pre-)crossed modules.
\end{definition}

\begin{proposition} \label{p1}
Assume we are given a prounipotent finite type presentation \eqref{5} then
$Kerd_0 \xrightarrow{d_1} F_u(X)$
is a free prounipotent pre-crossed module on the set $Y$.
\end{proposition}
\begin{proof}
For each function $\nu:Y\rightarrow A(\mathbb{Q}_p),$
where $A\xrightarrow{\partial} G$ is a prounipotent pre-crossed module,
let us construct the $\psi$ - homomorphism in the following diagram, with
$\widetilde{\partial}(A)=\partial(A)$, $\widetilde{\partial}(G)=id_G$, $\widetilde{\partial}(a,b)=\partial(a)\cdot b$
for $a\in A(\mathbb{Q}_p), b\in G(\mathbb{Q}_p)$.
$$
\xymatrix{
{Ker d_0} \ar[r] \ar@/^/[dr]^{\varphi} \ar@/^1pc/[rr]^{d_1} & F_u(X \cup Y)  \ar[r]^{d_0} \ar[d]^{\psi} & F_u(X) \ar[d]^{f}
\\
Y \ar[r]^{\nu} \ar[ru] & A\leftthreetimes G \ar[r]^{\widetilde{\partial}} & G}
$$
The semidirect product $A \leftthreetimes G$ \cite[7i]{Milne} is a prounipotent group
(prounipotence follows from the fact that a semidirect product of nilpotent Lie algebras is a nilpotent Lie algebra),
since it is constructed from a prounipotent action, which is given in the pre-crossed module $(A,G,\partial)$.
The fact that $\widetilde{\partial}$ is a homomorphism follows by direct computation using the fact that
for $\partial$ the property СМ 1) from Definition \ref{d1} holds.

Set $\psi(\mathbb{Q}_p)$ on $X$ as the composite $f(\mathbb{Q}_p)d_1(\mathbb{Q}_p)$, and set $\psi(\mathbb{Q}_p)$ on $Y$
equal to $\nu(Y).$ The universal property of $F_u(X\cup Y)$ yields a homomorphism of prounipotent groups
$\psi: F_u(X\cup Y)\rightarrow A\leftthreetimes G.$ Put $\varphi:Kerd_0\rightarrow A\leftthreetimes G$
equal to $\varphi=\psi\mid_{Kerd_0}$, i.~e. restriction of $\psi$ to $Kerd_0$.
\end{proof}

Constructing the prounipotent crossed module of the presentation $$ (C_u,F_u(X),\overline{d_1})$$
from the prounipotent presentation is standard and made by compilation of
the constructions from \cite{BH,Mikh2012}:
$$\frac{Kerd_0}{P_u} \xrightarrow{\overline{d_1}} F_u(X),$$ here $P_u=<{}^{d_1(b)}a=bab^{-1}>$
is the normal Zariski closure of the subgroup of Peifer commutators.
By the homotopy groups of a simplicial group $F_{\bullet}$ one calls the homology groups of
its Moore complex $(NF_n=\cap^{n-1}_{i=0}kerd_i, d_n|_{NF_n})$.
In \cite[5.7]{BL} (\cite[4.3.8]{Por},\cite[page 23]{Mikh2012}) it is proved that in any
simplicial group degenerate in dimension two and in particular in the simplicial group $F_{\bullet}$ \eqref{2},
arising from the presentation \eqref{eq1}, one has the coincidence of subgroups $Im\overline{d_2}=P_u,$
where $\overline{d_2}:NF_2\rightarrow NF_1,$ and hence $\pi_1(F_{\bullet})\cong ker\overline{d_1}$
(recall that in the discrete case $\pi_1(F_{\bullet})$ coincides with $\pi_2$
of the standard 2-complex of the presentation \eqref{eq1} of the group).

\begin{lemma} \label{l3}
Assume we are given a prounipotent finite type presentation \eqref{5}, then one has the
isomorphism of prounipotent crossed modules arising from coincidence of subgroups $P_u$ and $[Kerd_0, Kerd_1]$ in
$Kerd_0$,
$$ (C_u,F_u(X),\overline{d_1})=(\frac{Kerd_0}{P_u}, F_u(X),
\overline{d_1})\cong(\frac{Kerd_0}{[Kerd_0, Kerd_1]},F_u(X),\overline{d_1}).$$
\end{lemma}
\begin{proof}
The proof consists of two steps:

1) using elementary substitutions we see that Peifer commutators are representable in the form of
commutators $[x^{-1}s_0d_1(x),y],$ where $x,y\in Ker(d_0)$;

2) elements of $Ker(d_1)$ are exactly the elements  $x^{-1}s_0d_1(x),$ where $x\in Ker(d_0).$

For details, see \cite[p.23]{Mikh2012}.
\end{proof}

Below, by the crossed module of a prounipotent presentation \eqref{5} we shall understand
the crossed module from Lemma \ref{l3}. Recall the notation $$R_u=im(\overline{d_1}).$$

\begin{definition} \label{d3}
Let $A$ be a complete linearly compact Hopf algebra over a field $k$ or a topological group
(the field is considered with discrete topology).
By a left (or right) complete topological $A$-module we shall call a linearly compact topological
$k$- vector space $M$ with a structure of $A$-module
such that the corresponding $k$-linear action $ A\widehat{\otimes} M\rightarrow M$
is continuous. We assume that the topology on $M$ is given by a fundamental system of neighborhoods of zero
$M=M^0\supseteq M^1 \supseteq M^2\supseteq    \ldots,$ where $M^j$ are topological $A$-submodules in
$M$ of finite codimension (finiteness of type) and $M\cong \varprojlim M/M^j.$
By a homomorphism of topological $A$-modules one calls a continuous $A$- module homomorphism.

If the filtration $M^j$ admits a compression such that for each $j$
the compression quotients $M^{j_i}_i/M^{(j+1)_{i}}_{i+1}$ are trivial $A$- modules (i.~e. the action of $A$ is trivial),
then we shall call such topological $A$-module prounipotent.
 \end{definition}

In contrast to \cite{Hain1992}, in the definition of a topological module we always require ``finiteness of type''.
If $A$ is a complete linearly compact Hopf algebra, then the corresponding category of
topological $A$-modules is Abelian \cite[Theorem 3.4]{Hain1992}.

Let $G$ be a finitely generated group and $k$ a field $char(k)=0$, and let $G_u$ be its prounipotent completion.
The lower central series $C_n(G_u)$ of the prounipotent group $G_u=\varprojlim G_{\alpha}$ ($G_{\alpha}$
are unipotent groups)
is defined by the rule $C_n(G_u)=\varprojlim C_n(G_{\alpha})$.
Since $G_{\alpha}$ are linear algebraic groups, there is an opportunity to define
$C_n(G_{\alpha})$ as subgroups of the lower central series of the linear algebraic group
(details see in \cite{HM2003}). Denote the $n$-th element of the lower central series of a group $G$ by $L_n(G)$.

The prounipotent completion of any group is the inverse limit of unipotent completions of its
nilpotent torsion free quotients $$G^{\wedge}_u = \varprojlim_n (G/D_n)^{\wedge}_u\cong\varprojlim_n (G_u/C_n(G_u)),$$
where \cite[Chapter 11, Theorem 1.10]{Pas} $D_n=\sqrt{L_n(G)}=\{x\in G|x^n\in L_n(G), n\geq1\}.$

By analogy with discrete and pro-$p$ cases, on the rational points of the relations module
$\overline{R_u}(\mathbb{Q}_p)=R_u/[R_u,R_u](\mathbb{Q}_p)\cong R_u(\mathbb{Q}_p)/[R_u(\mathbb{Q}_p),R_u(\mathbb{Q}_p)]$
of a prounipotent presentation \eqref{5} there is a structure of topological
$\mathcal{O}(F_u)^*$- module with the filtration
$$\overline{R_u}(\mathbb{Q}_p)^j=\frac{[R_u(\mathbb{Q}_p),R_u(\mathbb{Q}_p)C_j(F_u(X))(\mathbb{Q}_p)]}
{[R_u(\mathbb{Q}_p),R_u(\mathbb{Q}_p)]}.$$
Similarly, for $\overline{C_u}(\mathbb{Q}_p)$ the filtration is given using the rule
$$\overline{C_u}(\mathbb{Q}_p)^j=\frac{[kerd_0(\mathbb{Q}_p),kerd_0(\mathbb{Q}_p)C_j(F_u(X\cup Y))(\mathbb{Q}_p)]}
{[kerd_0(\mathbb{Q}_p),kerd_0(\mathbb{Q}_p)kerd_1(\mathbb{Q}_p)]}.$$

\begin{proposition} \label{p2}
Assume we are given a prounipotent finite type presentation \eqref{5}, then
$\overline{C}_u(\mathbb{Q}_p)$ and $\overline{R}_u(\mathbb{Q}_p)$
are prounipotent topological $\mathcal{O}(G_u)^*$- modules
with the filtrations $\overline{C_u}(\mathbb{Q}_p)^j$ and $\overline{R_u}(\mathbb{Q}_p)^j$.
\end{proposition}

\begin{proof} Let us give the proof for $\overline{C}_u(\mathbb{Q}_p),$ for
$\overline{R}_u(\mathbb{Q}_p)$ the argument is completely similar.

The fact that $\cap C_j(F_u(X\cup Y))(\mathbb{Q}_p)=1$ implies $\cap \overline{C_u}(\mathbb{Q}_p)^j=0.$
CM 2) implies that $R_u$ acts on $\overline{C}_u$ trivially, and hence $\overline{C}_u(\mathbb{Q}_p)$
is a $\mathcal{O}(G_u)^*$-module (for details see \cite[2.4]{BH} or \cite[3.2.5]{Por}).
Finite dimension of quotients $\overline{C_u}(\mathbb{Q}_p)^j/\overline{C_u}(\mathbb{Q}_p)^{j+1}$
follows from finiteness of the presentation.
Indeed, since $\overline{C_u}(\mathbb{Q}_p),$ as an Abelian prounipotent group,
is generated by the elements $\cup_{y\in Y} G_u(\mathbb{Q}_p)\cdot y$
(the finite set $Y$ corresponds to the relations of the presentation \eqref{5}), one has an
epimorphism of $\mathcal{O}(G_u)^*$-modules $(\mathcal{O}(G_u)^*)^{|Y|}\twoheadrightarrow\overline{C_u}(\mathbb{Q}_p)$.
The action by conjugation is prounipotent since $F_u(X\cup Y)$ is a prounipotent group.
\end{proof}

In \cite[A.2]{HM2003} it is proved that the continuous prounipotent completion of a
finitely generated free pro-$p$-group with a basis $X$
and the prounipotent completion of the discrete free group contained in it with the same basis
(and the induced pro-$p$-topology) are naturally isomorphic.
This follows from the fact that any homomorphism from a finitely generated pro-$p$-group into
the group of $\mathbb{Q}_p$-points of a unipotent group is always continuous \cite[Lemma A.7.]{HM2003}.
Hence for any finite $X$ we have $F(X)^{\wedge}_w=F_u(X),$ where $F_u(X)$ is the free
prounipotent group \cite{Vez}. Since $d_0, d_1$ are epimorphisms, $d_1(kerd_0)$ is a normal subgroup in $F_u$
the quotient map $F_u\rightarrow G_u=F_u/d_1(kerd_0)$ is a coequalizer \cite[3.3]{Mac}
of the diagram $d_0, d_1$. The functor of continuous prounipotent completion has
a right adjoint functor ($\mathbb{Q}_p$-points of a prounipotent group) and hence preserves
coequalizers \cite[5.5]{Mac}, and hence for finite type presentations one has an
isomorphism (similar idea can be found in \cite[p.284]{Qui5})
\begin{equation}
G^{\wedge}_w\cong G_u, \label{eq3}
\end{equation}
where $G^{\wedge}_w$ is the continuous prounipotent completion of the pro-$p$-group from
presentation \eqref{2}, here $G_u$
is the prounipotent group from \eqref{5} and the presentation \eqref{5} is obtained from \eqref{2}
by means of prounipotent completion of pro-$p$-groups $F(X)$ and $F(X\cup Y)$.

The following results describe the structure of Abelianization of continuous prounipotent completions
of the crossed modules of $QR$ pro-$p$-presentations announced in \cite{Mikh2014}.

\begin{lemma} \label{l4}
Let \eqref{eq1} be a finite $QR$ pro-$p$-presentation, then using the isomprphism
$\overline{R^{\wedge}_w}(\mathbb{Q}_p)\cong \overline{R} \widehat{\otimes}\mathbb{Q}_p=
\varprojlim_n R/[R,R \mathcal{M}_n]\otimes \mathbb{Q}_p$ of Abelian prounipotent groups,
$\overline{R^{\wedge}_w}(\mathbb{Q}_p)$ is endowed with a structure of topological $G$-module.
\end{lemma}
\begin{proof}
The idea of construction is contained in \cite[Theorem 1]{Mikh2014},
let us expose the construction completely. Let us check that the topological vector space
$\overline{R} \widehat{\otimes}\mathbb{Q}_p$, considered as the group of $\mathbb{Q}_p$-points of
an Abelian prounipotent group, has the universal property inherent to the group of $\mathbb{Q}_p$-points
of $\overline{R^{\wedge}_w}$, with respect to continuous (in the pro-$p$-topology induced from $F(X)$)
Zariski-dense homomorphisms of $R$ into Abelian prounipotent groups.
Since each prounipotent group is the projective limit of a surjective projective system
of unipotent groups, it suffices to check the universal property for homomorphisms into Abelian
unipotent groups.

Let $\phi : \overline{R}\rightarrow U(\mathbb{Q}_p)\cong \mathbb{Q}_p^l$ be a continuous in
pro-$p$-topology and Zariski dense homomorphism.
Since $\phi (\overline{R})$ is dense in $\mathbb{Q}_p^l$ in the Zariski topology and $\mathbb{Z}_p$
is a principal ideal domain, then, by compactness of $\overline{R}$, we obtain
$\phi(\overline{R})\cong \mathbb{Z}_p^l$. Let
$\gamma_n:\mathbb{Z}_p^l\twoheadrightarrow \mathbb{Z}_p^l/p^n\mathbb{Z}_p^l=\mathbb{Z}^l/p^n\mathbb{Z}^l,
W=ker(\gamma_n\circ\phi).$

By construction of the topology on $R$, we can find $k\in \mathbb{N}$ such that
$R \cap \mathcal{M}_k\subseteq W$ \cite{Koch}, but $[R,\mathcal{M}_k]\subseteq R\cap \mathcal{M}_k.$
$R/[R,R\mathcal{M}_k]$ is a free Abelian pro-$p$-group (due to quasirationality), hence we can
define a homomorphism of free Abelian pro-$p$-groups
$ \nu_k:R/[R,R\mathcal{M}_k]\twoheadrightarrow \phi(R)\cong\mathbb{Z}_p^l$ by the formula
$\nu_k(b_i)=\phi(a_i)$ on a free basis $b_i\in R/[R,R\mathcal{M}_k], i\in I$.
Let us describe the construction of this basis in more detail.

$$\xymatrix{
  \overline{R} \ar[d]_{\gamma_k} \ar[r]^{\phi}
                & \phi(R) \ar[d] \\
  \frac{R}{[R,R\mathcal{M}_k]} \ar@{.>}[ur]|-{\kappa_k} \ar[r]_{\nu_k}
                & \frac{\phi(R)}{\Phi(\phi(R))}}$$

1) in $\overline{R}$ there exists a convergent basis $\{a_j\},j\in J$ \cite[2.4.4]{ZR};

2) if $\{\widetilde{a}_i\},i\in I, |I|=l$ is a basis in $\phi(R)$, then $\{a_j\},j\in J, J\supseteq I$
can be chosen so that $\phi(a_i)=\widetilde{a}_i,\forall i\in I$ \cite[7.6.10]{ZR}, and the rest elements
of the basis generate a subgroup orthogonal to $\langle\{a_i\},i\in I\rangle$ \cite[3.3.8(c)]{ZR};

3) then we set $\gamma_k(a_i)=b_i$ as part of the system of free generators $R/[R,R\mathcal{M}_k]$,
because they are linearly independent in $\phi(R)/\Phi(\phi(R)),$ where $\Phi(\phi(R))$
is the Frattini subgroup in $\phi(R)$, and hence in $R/R^p[R,R\mathcal{M}_k]$;

4) now $\kappa_k$ is defined in the unique way putting $\kappa_k(b_i)=\widetilde{a}_i, i\in I,$
and $\kappa_k(b_{\lambda})=0$ on the rest of free generators of $R/[R,R\mathcal{M}_k]$.

We construct
$\psi_k : \overline{R}\widehat{\otimes}\mathbb{Q}_p\rightarrow U (\mathbb{Q}_p)= \mathbb{Q}_p^l,$
extending $\phi,$ by the rule
$\psi_k=(\kappa_k\otimes id)\circ pr_k,$
where
$pr_k: \overline{R} \widehat{\otimes} \mathbb{Q}_p \twoheadrightarrow  \frac{R}{[R,R\mathcal{M}_k]}\otimes \mathbb{Q}_p$
is the projection. Let
$$ker(\gamma_n)=ker(\frac{R}{[R,R]}\xrightarrow{\gamma_n} \frac{R}{[R,R\mathcal{M}_n]}),$$
$$ker(\gamma_n\widehat{\otimes}\mathbb{Q}_p)=ker(\frac{R}{[R,R]}\widehat{\otimes}\mathbb{Q}_p
\xrightarrow{\gamma_n\widehat{\otimes} id} \frac{R}{[R,R\mathcal{M}_n]}\otimes \mathbb{Q}_p).$$

Now check that $ker(\gamma_n\widehat{\otimes}\mathbb{Q}_p)$ provide a $G$-filtration,
which will be of finite type, due to finiteness of considered presentations.
First of all note that $ker(\gamma_n\widehat{\otimes}\mathbb{Q}_p)\cong \varprojlim_{k\geq n }
[R,R\mathcal{M}_n]/[R,R\mathcal{M}_k]\otimes \mathbb{Q}_p.$
 Indeed, for each $k\geq n$ we have a short exact sequence
$$0\rightarrow \frac{[R,R\mathcal{M}_n]}{[R,R\mathcal{M}_k]}\rightarrow \frac{R}{[R,R\mathcal{M}_k]}\rightarrow
\frac{R}{[R,R\mathcal{M}_n]}\rightarrow 0$$
of finitely generated free Abelian pro-$p$-groups (due to quasirationality), hence
the following sequence of finite dimensional vector spaces will be also exact:
$$0\rightarrow \frac{[R,R\mathcal{M}_n]}{[R,R\mathcal{M}_k]}\otimes \mathbb{Q}_p\rightarrow \frac{R}
{[R,R\mathcal{M}_k]}\otimes \mathbb{Q}_p\rightarrow \frac{R}{[R,R\mathcal{M}_n]}\otimes \mathbb{Q}_p\rightarrow 0,$$
but since on the right we have a homomorphism of $G$-modules, on the left we have also a $G$-module.
Now, since $\varprojlim_k^1 [R,R\mathcal{M}_n]/[R,R\mathcal{M}_k]\otimes \mathbb{Q}_p=0$,
we have an exact sequence $$0\rightarrow\varprojlim_{k\geq n } \frac{[R,R\mathcal{M}_n]}{[R,R\mathcal{M}_k]}\otimes
\mathbb{Q}_p\rightarrow \frac{R}{[R,R]}\widehat{\otimes}\mathbb{Q}_p\rightarrow \frac{R}{[R,R\mathcal{M}_n]}\otimes
\mathbb{Q}_p\rightarrow 0,$$
which implies the required statement.
\end{proof}

\begin{lemma} \label{l5}
Assume we are given a finite type pro-$p$-presentation \eqref{eq1}, then $\mathbb{Q}_p$- points of the
Abelianization $\overline{C^{\wedge}_w}$ of the prounipotent group
$$C^{\wedge}_w:=(\frac{kerd_0}{[kerd_0,kerd_1}])^{\wedge}_w$$
have a structure of topological finite type $G$-module, given by the decomposition
$\overline{C^{\wedge}_w}(\mathbb{Q}_p)\cong\varprojlim_k(\frac{kerd_0}
{[kerd_0,kerd_1Kerd_0\mathcal{M}_k]}\otimes\mathbb{Q}_p)
\cong\varprojlim_k(\mathbb{Q}_p[\frac{G}{\mathcal{M}_k}])^{|Y|}.$

\end{lemma}
\begin{proof}
As noted above, the topology on $Kerd_0$ is induced by the topology on $F(Y\cup X)$.
Then the argument of Lemma \ref{l4} (with small deviations) yields a presentation
$(\overline{Kerd_0})^{\wedge}_w\cong \varprojlim_k (\overline{F(Y\times F(X)/\mathcal{M}_k))^{\wedge}_w},$
and hence, identifying $\overline{C^{\wedge}_w}$ with the projective limit of linear algebraic groups,
the group of $\mathbb{Q}_p$-points of $\overline{C^{\wedge}_w}(\mathbb{Q}_p)$, we obtain
$$\overline{C^{\wedge}_w}(\mathbb{Q}_p)\cong \varprojlim_k
(\frac{kerd_0}{[kerd_0,kerd_1kerd_0\mathcal{M}_k]})^{\wedge}_w(\mathbb{Q}_p)\cong \varprojlim_k
(\mathbb{Q}_p(G/\mathcal{M}_k))^{|Y|}.$$
The latter isomorphism is a consequence of the general theory of crossed modules of pro-$p$-presentations.
Indeed, \cite[Theorem 1.13, Proposition 14]{Mikh2014} gives an isomorphism of $G$-modules
$kerd_0/[kerd_0,kerd_1kerd_0]\cong \mathbb{Z}_pG^{|Y|}.$
It remains to define, similarly to Lemma \ref{l4}, a finite type $G$-filtration by the rule
$$\varprojlim_{k,k\geq n}\frac{[kerd_0,kerd_1kerd_0\mathcal{M}_n]}{[kerd_0,kerd_1kerd_0\mathcal{M}_k]}\otimes
\mathbb{Q}_p.$$
\end{proof}

Consider the following diagram of Abelian prounipotent groups:
\begin{equation}
\xymatrix{
  \overline{C^{\wedge}_w} \ar[d]_{\kappa} \ar[r]^{\gamma} & \overline{R^{\wedge}_w} \ar[d]^{\tau} \\
  \overline{C_u} \ar[r]^{\mu} & \overline{R}_u} \label{eq2}
\end{equation}
We have denoted by $\overline{C^{\wedge}_w}, \overline{C_u}, \overline{R^{\wedge}_w}, \overline{R_u}$
the Abelianizations of the corresponding prounipotent groups.
The homomorphisms $\kappa$ and $\tau$ arise from the universal properties of
$\overline{C^{\wedge}_w}$ and $\overline{R^{\wedge}_w}$ with respect to the homomorphisms induced by
a continuous embedding of a pro-$p$- presentation \eqref{2} into the rational points of
the corresponding prounipotent presentation \eqref{5}
($kerd_0\hookrightarrow kerd_{0}(\mathbb{Q}_p)$ and $R\hookrightarrow R_u(\mathbb{Q}_p)$).
The homomorphism $\gamma$ is induced, in the notations \eqref{2}, by the homomorphism of pro-$p$-groups
$d_1:kerd_0\twoheadrightarrow R.$
\begin{theorem} \label{t01}
Let  \eqref{eq1} be a finite $QR$ pro-$p$-presentation of a pro-$p$- group $G$,
then one has a commutative diagram \eqref{eq2} of Abelian prounipotent groups in which
on $\mathbb{Q}_p$-points $\gamma(\mathbb{Q}_p)$ is a homomorphism of topological $G$-modules,
$\mu(\mathbb{Q}_p)$ is a homomorphism of $\mathcal{O}(G_u)^*$-modules, and $\kappa$ and $\tau$
are induced by embedding of a pro-$p$- presentation \eqref{2} into the corresponding
prounipotent presentation \eqref{5} and are $G$-homomorphisms on Zariski dense subgroups $\overline{R}$ and
$\frac{kerd_0}{[kerd_0,kerd_1kerd_0]}$ in $\overline{R^{\wedge}_w}(\mathbb{Q}_p)$
and $\overline{C^{\wedge}_w}(\mathbb{Q}_p)$, respectively.
\end{theorem}
\begin{proof}
Commutativity of the diagram follows from the fact that the morphisms in \eqref{5} are defined using
the morphisms in \eqref{2}.
Since the action both in the upper and the lower rows of the commutative diagram is defined by conjugation,
these are $G$- homomorphisms on dense subgroups $\frac{kerd_0}{[kerd_0,kerd_1kerd_0]}$ and $\overline{R}$,
where $G$ acts on $\overline{C_u}(\mathbb{Q}_p)$ and $\overline{R_u}(\mathbb{Q}_p)$
through the homomorphism into its prounipotent completion.
Continuity and the module homomorphism property of $\gamma$ are obvious from the construction
in Lemma \ref{l4} and Lemma \ref{l5}. The properties of $\mu$
follow from Proposition \ref{p2} and by general constructions from the beginning of this part of the paper.
Recall that since $R_u$ acts on $\overline{R_u}$ trivially, $\overline{R_u}$ is an $\mathcal{O}(G_u)^*$-module.

It is clear that $\kappa$ and $\tau$ determine continuous maps of topological vector spaces
in the introduced in Lemma \ref{l4} and Lemma \ref{l5} filtrations on
$\overline{R^{\wedge}_w}(\mathbb{Q}_p)$ and $\overline{C^{\wedge}_w}(\mathbb{Q}_p),$
and continuity is actually shown there.
\end{proof}

We have already noted (before Lemma \ref{l3}) that in any simplicial group degenerate in dimensions two
and greater, and in particular in the simplicial group \eqref{5}
one has an isomorphism $\pi_1(F_{\bullet})\cong ker\overline{d_1}.$
And hence for the prounipotent presentation \eqref{5} (in the notations of Lemma \ref{l3})
it is natural to introduce the second homotopy group of the presentation as
$$u_2(\mathbb{Q}_p)=ker\{ C_u(\mathbb{Q}_p)\rightarrow R_u(\mathbb{Q}_p) \},$$
and also for the $QR$ pro-$p$-presentation \eqref{2} the continuous prounipotent
completion of its second homotopy group
$${\pi_2}\widehat{\otimes}\mathbb{Q}_p=\varprojlim \frac{\pi_2}{\pi_2\cap
\frac{[kerd_0,kerd_1kerd_0\mathcal{M}_k]}{[kerd_0,kerd_1kerd_0]}}\otimes\mathbb{Q}_p.$$
The next Corollary yields a relation between the second homotopy group of
the $QR$-pro-$p$-presentation $\pi_2=ker \{ d_1:C \to R \}$ and the diagram from Theorem \ref{t01}.

\begin{corollary} \label{s0}
In the notations of the previous Theorem, one has the diagram of
topological vector spaces
$$\xymatrix{
  {\pi_2}\widehat{\otimes}\mathbb{Q}_p \ar[d]_{} \ar[r]^{} & \overline{C^{\wedge}_w}(\mathbb{Q}_p)
  \ar[d]_{\kappa(\mathbb{Q}_p)} \ar[r]^{\gamma(\mathbb{Q}_p)} & \overline{R^{\wedge}_w}(\mathbb{Q}_p)
  \ar[d]^{\tau(\mathbb{Q}_p)} \\
   u_2(\mathbb{Q}_p)\ar[r]^{} & \overline{C_u}(\mathbb{Q}_p)
   \ar[r]^{\mu(\mathbb{Q}_p)} & \overline{R_u}(\mathbb{Q}_p)   }$$
with exact rows, and an isomorphism of $\mathcal{O}(G_u)^*$-modules
$\overline{C_u}(\mathbb{Q}_p)\cong {\mathcal{O}(G_u)^*}^{|Y|}.$
\end{corollary}

\begin{proof}
Actually, generalizing \cite{LM1982}, one can see that $R_u$ is a free prounipotent group,
and hence the canonical surjection $C_u\twoheadrightarrow R_u$ splits, and we can use the same
argument as \cite[Proposition 4]{BH}, taking into account Theorem on commutant closure \cite[4.3]{Wat}.
In particular, compiling \cite[Proposition 14]{Mikh2012} using the prounipotent analog of
the Brown--Loday Lemma (\cite[5.7]{BL}, \cite[4.3.8]{Por}), we see that $u_2(\mathbb{Q}_p)$
is the prounipotent analog of the second homotopy group of the presentation.
Coincidence of $ker\mu(\mathbb{Q}_p)$ with ${\pi_2}\widehat{\otimes}\mathbb{Q}_p$ follows from
quasirationality and the fact that $\varprojlim^{1}=0$
for a surjective projective system of finite dimensional vector spaces.

The standard isomorphism $\overline{C_u}(\mathbb{Q}_p)\cong {\mathcal{O}(G_u)^*}^{|Y|}$ \cite[Proposition 29]{Por}
follows from Proposition \ref{p2}, since $\overline{C_u}(\mathbb{Q}_p)$ possesses the universal property
of a direct sum of algebras with respect to homomorphisms into prounipotent $\mathcal{O}(G_u)^*$-modules.
Details can be completely reconstructed from the arguments in \cite[Proposition 29]{Por}
using Proposition \ref{p1}.
\end{proof}

\begin{proposition}\label{p01}
Let $G$ be a finitely generated pro-$p$-group given by a presentation \eqref{eq1} with one relation
$r\neq1$, and assume that the natural homomorphism from $G$ to the $\mathbb{Q}_p$-points
$G^{\wedge}_w(\mathbb{Q}_p)$ of its prounipotent completion is an embedding,
then $cd(G)=2.$
\end{proposition}
\begin{proof} First, note that the condition $G\hookrightarrow G_w^{\wedge}(\mathbb{Q}_p)$
implies that the homomorphism $\kappa(\mathbb{Q}_p)$ from the commutative diagram (Corollary \ref{s0})
$$\xymatrix{
  {\pi_2}\widehat{\otimes}\mathbb{Q}_p \ar[d]_{} \ar[r]^{} & \overline{C^{\wedge}_w}(\mathbb{Q}_p)
  \ar[d]_{\kappa(\mathbb{Q}_p)} \ar[r]^{\gamma(\mathbb{Q}_p)} & \overline{R^{\wedge}_w}(\mathbb{Q}_p)
  \ar[d]^{\tau(\mathbb{Q}_p)} \\
   u_2(\mathbb{Q}_p)\ar[r]^{} & \overline{C_u}(\mathbb{Q}_p) \ar[r]^{\mu(\mathbb{Q}_p)} & \overline{R_u}(\mathbb{Q}_p)
  }$$
is an isomorphism. Actually, since the image $\kappa(\mathbb{Q}_p)$ is dense, $\kappa(\mathbb{Q}_p)$
is an epimorphism. On the other hand, Lemma \ref{l5} implies that the elements $G\cdot y,$ where $y\in Y$
(in the notation of Definition \ref{2}) form a pseudo-basis of
$\overline{C^{\wedge}_w}(\mathbb{Q}_p)$ as a topological vector $\mathbb{Q}_p$-space.
But Corollary \ref{s0} implies that $\overline{C_u}(\mathbb{Q}_p)\cong {\mathcal{O}(G_u)^*}^{|Y|}$
and therefore, as a consequence, $\kappa(\mathbb{Q}_p)$ is injective.

Assume that $\overline{R}_u(\mathbb{Q}_p)\cong \mathcal{O}(G_u)^*$  \cite[Corollary 12]{Mikh2015},
then our diagram (only as a diagram of topological vector spaces) takes the form
\begin{equation}
\xymatrix{
  \mathcal{O}(G_u)^* \ar[d]_{\kappa(\mathbb{Q}_p)} \ar[r]^{\gamma(\mathbb{Q}_p)} & \overline{R}^{\wedge}_w(\mathbb{Q}_p)
  \ar[d]^{\tau(\mathbb{Q}_p)} \\
  \mathcal{O}(G_u)^* \ar[r]^{\mu(\mathbb{Q}_p)} & \mathcal{O}(G_u)^*.}
\end{equation}
Due to quasirationality of presentations of pro-$p$-groups with a single defining relation \cite[Proposition 1]{Mikh2014},
Lemma \ref{l4} gives an isomorphism
$$\overline{R^{\wedge}_w}(\mathbb{Q}_p)\cong \overline{R} \widehat{\otimes}\mathbb{Q}_p=\varprojlim_n
R/[R,R \mathcal{M}_n]\otimes \mathbb{Q}_p.$$ Then left exactness of the functor
$\varprojlim$ yields injectivity of
$\overline{R}\hookrightarrow \overline{R}\widehat{\otimes}\mathbb{Q}_p$, which implies
$G$-embedding $\overline{R} \hookrightarrow \overline{R}^{\wedge}_w(\mathbb{Q}_p)$.
Therefore, $\overline{R} \hookrightarrow \mathcal{O}(G_u)^*,$ and hence the elements of the form
$G\cdot \overline{r}$ form a permutational basis in $\overline{R}.$ But then $G$ is an aspherical
pro-$p$-group \cite[\S2]{Mel1}, and since it is torsion free, one has $cd(G)= 2$ \cite[(3.2)]{Mel1}.
Let us note that $G^{\wedge}_w(\mathbb{Q}_p)\cong G_u(\mathbb{Q}_p)$ \eqref{eq3}, and hence
the statement of Proposition \ref{p01} one could equivalently assume $G\hookrightarrow G_u(\mathbb{Q}_p).$
\end{proof}

The following Corollary generalizes and explains the known group theory results
(see, for example, \cite{Lab} and \cite{Rom},
where absence of divisors of 0 in the graded algebra of a filtration is used).
We shall say that a pro-$p$-group $G$ is $p$-regular if for any finite quotient $G/\mathcal{M}_n(G)$
there exists a torsion free nilpotent quotient $G/V$ and $V\subset\mathcal{M}_n(G)$.
We shall also say that a pro-$p$-group $G$ has a $p$-regular filtration
$\mathcal{V}=(V_n,n \in \mathbb{N})$ if for any finite quotient
$G/\mathcal{M}_k(G)$ there exists $m(k) \in \mathbb{N}$ such that $V_{m(k)}\subset\mathcal{M}_k(G)$
and the quotients of this filtration are torsion free.

\begin{corollary} \label{s4}
Assume that a pro-$p$-group $G$ with one relation is $p$-regular, then $cd(G)= 2$.
In particular, if $G$ has a $p$-regular filtration $\mathcal{V}$, then $cd(G)= 2$.
\end{corollary}

\begin{proof}
If $G_{\lambda}$ is a finitely generated nilpotent pro-$p$-group without torsion, then, due to
\cite[Corollary A.4]{HM2003}, $G_{\lambda}$ is included into the group of $\mathbb{Q}_p$-points of a
unipotent group $U$ over $\mathbb{Q}_p,$ and hence $G_{\lambda}$ is embedded into the group of
$\mathbb{Q}_p$-points of its continuous prounipotent completion $(G_{\lambda})^{\wedge}_w$.

Since $G$ is $p$-regular, one has $G\cong\varprojlim G_{\lambda},$ where $G_{\lambda}$ are
finitely generated nilpotent torsion free pro-$p$-groups. For each $\lambda$
we have an embedding $G_{\lambda}\hookrightarrow (G_{\lambda})^{\wedge}_w(\mathbb{Q}_p),$
and due to left exactness of the projective limit functor, we have also an embedding
$G\hookrightarrow \varprojlim (G_{\lambda})^{\wedge}_w(\mathbb{Q}_p).$
Obviously, there exists an epimorphism $\zeta: G^{\wedge}_w \twoheadrightarrow \varprojlim (G_{\lambda})^{\wedge}_w,$
and the diagram
$$\xymatrix{
  G \ar[dr]_{\gamma} \ar[r]^{\beta}
                & G^{\wedge}_w(\mathbb{Q}_p) \ar[d]^{\zeta(\mathbb{Q}_p)}  \\
                & \varprojlim (G_{\lambda})^{\wedge}_w(\mathbb{Q}_p)             }$$
is commutative, but since $\gamma$ is injective, $\beta$ is also en embedding. And we can apply
Proposition \ref{p01}.
Clearly, the absence of torsion on the graded quotients of a $p$-regular filtration $\mathcal{V}$
is sufficient for $p$-regularity of $G$.
\end{proof}

\appendix

\section*{Acknowledgments}
The author thanks A. S. Mishchenko and V. M. Manuilov for useful discussions and constructive criticism.


\begin{thebibliography}{0}

\bibitem{And}
 M.~Andre,
 Methode simpliciale en algebre homologique et algebre commutative, Lecture Notes in Math.,
 37,
 Springer-Verlag,
 Berlin,
 1967


\bibitem{BH}
 R.~Brown\, J.~Huebschmann,
 Identities among relations,
 Low-dimensional topology,
 London Math. Soc., Lecture Notes Series,
 48,
 1982,
 153--202

\bibitem{BL}
 R.~Brown\,J.-L.~Loday,
 Van Kampen Theorems for diagrams of spaces,
 Topology,
 26,
 1987,
 311--337

\bibitem{Car}
P.Cartier,
A primer of Hopf algebras. Frontiers in Number Theory, Physics, and Geometry II.,
2007,537--615


\bibitem{DNR}
 S.~D$\check{a}$sc$\check{a}$lescu\, C.~N$\check{a}$st$\check{a}$sescu\, S.~Raianu\,
 Hopf Algebras. An Introduction. Pure and Applied Mathematics,
 235,
 Marcel Dekker,
 2001

 \bibitem{DM}
 P.~Deligne\, J. Milne\,
 Tannakian categories, in Hodge cycles, motives, and Shimura varietes, Springer Lecture Notes in Math. 900, 1982, pp. 101--228

\bibitem{Die}
 J.~Diedonne,
 Introduction to the theory of formal groups,
 Marcel Dekker Inc.,
 1973

\bibitem{Gru}
K. ~Gruenberg,
Relation modules of finite groups,
RCSM,
 25,
AMS,
1976

\bibitem{Hain1992}
 R.~Hain,
 Algebraic cycles and variations of mixed Hodge structure,
 Complex Geometry and Lie Theory, Proc. Symp. Pure Math.,
 53,
 1991, pp. 175--221

\bibitem{HM2003}
 R.~Hain\, M.~Matsumoto,
 Weighted completion of Galois groups and Galois actions on the fundamental group of $\mathbb{P}^1-\{0,1, \infty\}$,
 Compos. Math.,
 139,
 2,
 2003, pp. 119--167


\bibitem{KPT}
 L.~Katzarkov\, T.~Pantev\, B.To\"{e}n,
 Algebraic and topological aspects of the schematization functor,
 Compositio Math.,
 145,
 2009,
 633–-686

 \bibitem{Koch}
 H.~Koch\,
 Galois theory of $p$-extensions,
 Springer,
 2002

\bibitem{Lab}
 J.~Labute,
 Alg\`{e}bres de Lie et pro-p-groupes d\'{e}finis par une seule relation,
 Inventiones mathematicae,
 4,
 2,
 1967, pp. 142--158


\bibitem{LM1982}
 A.~Lubotzky\, A. Magid,
 Cohomology of unipotent and prounipotent groups,
 J.Algebra,
 74,
 1982, pp. 76--95

 \bibitem{Mac}
 S.~Mac Lane\,
 Categories for the working mathematician,
 Springer Verlag,
 1971

 \bibitem{Mel1}
 O. V. Mel'nikov, ``Aspherical pro-pp-groups'', Sb. Math., 193:11 (2002), 1639--1670



\bibitem{Mikh2017}
 A.\, M.~Mikhovich,
Identity Theorem for pro-$p$-groups,
 arXiv:1703.02996


\bibitem{Mikh2012}
 A.\, M.~Mikhovich
 Homotopy of profinite groups,
 arXiv:1205.4365

\bibitem{Mikh2017a}
A.\, M.~Mikhovich,
Quasirationality and aspherical (pro-$p$)presentations
arXiv:1704.05515

\bibitem{Mikh2014}
 A.~Mikhovich
 Quasirational relation modules and $p$-adic Malcev completions,
 Topol. Appl.
 201,
 2016,
 86--91

\bibitem{Mikh2015}
 A.~Mikhovich,
 Proalgebraic crossed modules of quasirational presentations,
 Trends in Mathematics,
 5,
 Bikkh\"{a}user,
 2016,
 109--114, arXiv:1507.03155v4


\bibitem{Milne}
 J.~Milne,
 Algebraic groups, Lie groups and their arithmetic subgroups, http://www.jmilne.org/math/CourseNotes/ALA.pdf,
 2011



\bibitem{Pas}
 D. Passman,
 The algebraic structure of group rings,
 Pure and Applied Mathematics, Wiley-Interscience,
 1977


\bibitem{Por}
 T.~Porter,
 Profinite Algebraic Homotopy,
 http://ncatlab.org/timporter/ files/ProfAlgHomotopy.pdf,
 2009

\bibitem{Pri2008}
 J.P.~Pridham,
 Pro-algebraic homotopy types,
 Proc. London Math. Soc.,
 97,
 2,
 2008,
 273--338


\bibitem{Pri2012}
 J.P.~Pridham,
 On the $l$-adic pro-algebraic and relative pro-$l$-fundamental groups,
 Arithmetics of Fundamental Groups, Contr. in Math. and Comp. Sciences,
 2,
 Springer,
 2012,
 245--279

\bibitem{Qui5}
 D.~Quillen,
 Rational homotopy theory,
 Annals of Math.,
 90,
 2,
 1969,
 205--295

\bibitem{ZR}
 L.~Ribes\, P.~Zalesskii,
 Profinite Groups,
 A Series of Modern Surveys in Mathematics,
 40,
 Springer,
 2000

\bibitem{Rom}
 N.~Romanovskii,
 On pro-p-groups with a single defining relator,
 Israel Journal of Mathematics,
 78,
 1,
 1992,
 65--73

 \bibitem{Se}
 J.\, P.~Serre\,
 Structure des certains pro-$p$-groupes (d'apres Demushkin),
 S\'eminaire Bourbaki,
 252,
 1--11,
 1962--1963


\bibitem{Vez}
 A.~Vezzani,
 The pro-unipotent completion,
 http://users.mat.unimi.it/users/vezzani/ Files/Research/prounipotent.pdf


\bibitem{Wat}
 W.~Waterhouse,
 Introduction to Affine Group Schemes,
 Springer,
 1979
\end{thebibliography}
\end{document}